\documentclass[12pt]{amsart}
\usepackage{array,CJK}
\usepackage{curves}
\usepackage{amstext}
\usepackage{amsfonts}
\usepackage{amsmath}
\usepackage{amssymb}
\def\cal{\mathcal}

\newtheorem{theorem}{Theorem}[section]
\newtheorem{corollary}[theorem]{Corollary}
\newtheorem{lemma}[theorem]{Lemma}

\theoremstyle{definition}

\newtheorem{example}[theorem]{Example}
\theoremstyle{remark}
\newtheorem{remark}[theorem]{Remark}

\title[dynamics on Orlicz spaces]
{dynamics of weighted translations on Orlicz spaces}
\author[C-C. Chen]{Chung-Chuan Chen}


\address{Department of Mathematics Education, National Taichung University of Education, Taichung 403, Taiwan}
\email{chungchuan@mail.ntcu.edu.tw}
\date{\today}
\begin{document}

\maketitle

\begin{abstract}
Let $G$ be a locally compact group, and let $\Phi$ be a Young function.
In this paper, we give sufficient and necessary conditions for weighted translation operators on
the Orlicz space $L^\Phi(G)$ to be chaotic and topologically multiply recurrent.
In particular, chaos implies multiple recurrence in our case.
\end{abstract}

\baselineskip17pt

\section{introduction}

In \cite{chen11,chen16,cc09,cc11}, we characterized chaotic, topologically transitive and topologically multiply recurrent
weighted translation operators on the Lebesgue space of a locally compact group,
which subsumes some works of weighted shifts on the space $\ell^p(\Bbb{Z})$ in \cite{cmp,cs04,ge00,sa95}.
Since then, some authors also contribute interesting results to linear dynamics on locally compact groups in \cite{ak17,aa17,kuchen17,hl16}.
Indeed, Azimi and Akbarbaglu recently extend the work of \cite{cc09,cc11} from the Lebesgue space to the Orlicz space in \cite{aa17}
where they give a sufficient and necessary condition for weighted translation operators to be topologically transitive on the Orlicz spaces
of locally compact groups. The Orlicz space is a type of function space generalizing the Lebesegue space.
Linear dynamics on Orlicz spaces have not attracted the attention of authors working in this filed except \cite{aa17}. In this note, we will continue the theme of the wider setting on Orlicz spaces, and study linear chaos and topologically multiple recurrence.

We recall an operator $T$ on a separable Banach space $X$ is called {\it topologically transitive} if given two nonempty open sets $U,V\subset X$,
there exists $n\in\Bbb{N}$ such that $T^nU\cap V\neq \emptyset$. If $T^nU\cap V\neq \emptyset$ from some $n$ onwards, then $T$ is {\it topologically mixing.} It is well known in \cite{gpbook} that topological transitivity and hypercyclicity are equivalent on separable Banach spaces. A linear operator $T$ is {\it hypercyclic} if there is $x\in X$ such that its orbit under $T$, denoted by
$Orb(T,x):=\{T^nx: n\in\Bbb{N}\}$ is dense in $X$. If $T$ is topologically transitive (hypercyclic) together with the dense set of periodic elements of $T$,
then $T$ is said to be {\it chaotic}. Linear chaos and hypercyclicity have been studied intensely during last three decades.
Indeed, Salas characterized hypercyclic bilateral weighted shifts on $\ell^p(\Bbb{Z})$ in \cite{sa95}.
The characterization for bilateral weighted shifts on $\ell^p(\Bbb{Z})$ to be Ces\`aro hypercyclic was given in \cite{ls02}.
Also, Costakis and Sambarino in \cite{cs04} gave a sufficient and necessary condition for bilateral weighted shifts on $\ell^p(\Bbb{Z})$ to be mixing.
In \cite{ge00}, Grosse-Erdmann characterized chaotic bilateral weighted shifts on $\ell^p(\Bbb{Z})$. Linear chaos and hypercyclicity on weighed $L^p(\Bbb{R})$ and $L^p(\Bbb{C})$ spaces, and other spaces were studied in \cite{bb09,BBCP,bmp03,chen141,DE01,DSW97,kalmes07,mp10}.
We refer to these classic books \cite{bmbook,gpbook,kosticbook} on this subject.

Among other related notions in topological dynamics, recurrence is close to, but weaker than topological transitivity (hypercyclicity) in \cite{cmp}. It is well known that every transitive (hypercyclic) operator is recurrent on separable Banach spaces in \cite{cp12}.
However, this is not the case for topologically multiple recurrence, which is a stronger notion than recurrence. Indeed, in \cite{cp12} there exists a transitive weighted backward shift on $\ell^2(\Bbb Z)$, which is not topologically multiply recurrent.
We recall an operator $T$ is {\it topologically multiply recurrent} if for every positive integer $L$ and every nonempty open subset $U$ of $X$, there is some $n\in\Bbb{N}$ such that $U\cap T^{-n}U\cap T^{-2n}U\cap\cdot\cdot\cdot\cap T^{-Ln}U\neq\emptyset.$
If $L=1$, then $T$ is called {\it recurrent}, that is, the condition $U\cap T^{-n}U\neq\emptyset$ is satisfied.
In \cite{cp12}, Costakis and Parissis characterized topologically multiply recurrent weighted shifts
on $\ell^p(\Bbb{Z})$ in terms of the weight sequence. Since the discrete group $\Bbb{Z}$ is a special case of locally compact groups, we recently also investigated this notion in the setting of locally compact groups in \cite{chen16,chen162}.

A continuous, even and convex function $\Phi:{\Bbb R}\rightarrow {\Bbb R}$ is called a {\it Young function} if it satisfies $\Phi(0)=0$, $\Phi(t)>0$ for $t>0$, and $\lim_{t\rightarrow\infty}\Phi(t)=\infty$.
For a Young function $\Phi$, the complementary function $\Psi$ of $\Phi$ is given by
$$\Psi(y)=\sup\{x|y|-\Phi(x):x\geq 0\}\qquad(y\in\Bbb{R}),$$
which is also a Young function. If $\Psi$ is the complementary function $\Phi$, then
$\Phi$ is the complementary function $\Psi$, and they satisfy the Young inequality
$$xy\leq \Phi(x)+\Psi(y)\qquad(x,y\geq0).$$

Let $G$ be a locally compact group with identity $e$ and a right Haar measure $\lambda$. Then the {\it Orlicz space} $L^\Phi(G)$ is defined by
$$L^\Phi(G)=\left\{f:G\rightarrow \Bbb{C}: \int_G\Phi(\alpha|f|)d\lambda<\infty\ \text{for some $\alpha>0$} \right\}$$
where $f$ is a Borel measurable function. Then the Orlicz space is a Banach space under the Luxemburg norm $N_\Phi$
defined for $f\in L^\Phi(G)$ by
$$N_\Phi(f)=\inf\left\{k>0:\int_G\Phi\left(\frac{|f|}{k}\right)d\lambda\leq1\right\}.$$
The Orlicz space is a generalization of the usual Lebesegue space.
Indeed, if $\Phi(t)=\frac{|t|^p}{p}$, then $L^\Phi(G)$ is the Lebesegue space $L^p(G)$.
The interesting properties and structures of Orlicz spaces have been investigated intensely over the last several decades.
Indeed, weighted Orlicz algebras on locally compact groups were investigated in \cite{oo15}, which generalized the group algebras.
The properties $(T_{L^\Phi})$ and $(F_{L^\Phi})$ for Orlicz spaces $L^\Phi$ were studied by Tanaka in \cite{ta17} recently.
In \cite{pi17}, Piaggio also considered Orlicz spaces and the large scale geometry of Heintze groups.
Hence it is nature to tackle hypercyclicity and linear chaos on Orlicz spaces.
For more discussions and recent works on Orlicz spaces, see \cite{aa17,cl17,ha15,ta17,oo15,pi17,rr91}.

We note that a Banach space admits a hypercyclic operator if, and only if, it is separable and infinite-dimensional \cite{an97,bg99}. Hence we assume that $G$ is second countable and $\Phi$ is $\Delta_2$-regular in this paper.
A Young function is said to be $\Delta_2$-{\it regular} in \cite{rr91} if there exist a constant $M>0$ and $t_0>0$ such that $\Phi(2t)\leq M\Phi(t)$ for $t\geq t_0$ when $G$ is compact, and $\Phi(2t)\leq M\Phi(t)$ for all $t>0$ when $G$ is noncompact. For example, the Young functions $\Phi$ given by
$$\Phi(t)=\frac{|t|^p}{p}\quad(1\leq p<\infty),\qquad \mbox{and}\qquad\Phi(t)=|t|^\alpha(1+|\log|t||)\quad (\alpha>1)$$
are both $\Delta_2$-regular in \cite{aa17,rr91}. If $\Phi$ is $\Delta_2$-regular, then the space $C_c(G)$ of all continuous functions on $G$ with compact support is dense in $L^\Phi(G)$.

A bounded continuous function $w:G \rightarrow (0,\infty)$ is called a {\it weight} on $G$.
Let $a \in G$ and let $\delta_a$ be the unit point mass at $a$.
A {\it weighted translation} on $G$ is
a weighted convolution operator $T_{a,w}: L^\Phi(G)\longrightarrow L^\Phi(G)$ defined by
$$T_{a,w}(f)= wT_a(f) \qquad (f \in L^\Phi(G))$$
where $w$ is a weight on $G$ and $T_a(f)=f*\delta_a\in L^\Phi(G)$ is
the convolution:
$$(f*\delta_a)(x)= \int_{y\in G} f(xy^{-1})\delta_a(y) = f(xa^{-1})
\qquad (x\in G).$$
If $w^{-1}\in L^{\infty}(G)$, then we can define a self-map $S_{a,w}$ on $L^\Phi(G)$  by
$$S_{a,w}(h) = \frac{h}{w}* \delta_{a^{-1}} \qquad (h \in L^\Phi(G))$$ so that
$$T_{a,w}S_{a,w}(h) = h \qquad (h \in L^\Phi(G)).$$
We assume $w,w^{-1}\in L^\infty(G)$ throughout.

Since the weighted translation $T_{a,w}$ is generated by the group element $a$ and the weight $w$,
some elements $a\in G$ and weights $w$ should be excluded. For example,
if $\|w\|_{\infty}\leq1$, then $\|T_{a,w}\|\leq1$ and $T_{a,w}$ is never transitive (hypercyclic).
Also, it is proved in \cite{aa17} that
$T_{a,w}$ is not transitive if $a$ is a torsion element of $G$.
In both cases above, $T_{a,w}$ is never chaotic.
An element $a$ in a group $G$ is called a {\it torsion element} if
it is of finite order. In a locally compact group $G$, an element $a\in G$ is called {\it
periodic} (or {\it compact}) in \cite{cc11} if the closed subgroup $G(a)$
generated by $a$ is compact. We call an element in
$G$ {\it aperiodic} if it is not periodic. For discrete groups,
periodic and torsion elements are identical.

It is showed in \cite{cc11} that an element $a\in G$ is aperiodic if, and only if, for any compact set $K\subset G$, there exists some
$M\in\Bbb N$ such that $K\cap Ka^{\pm n}=\emptyset$ for all $n>M$.
We note that \cite{cc11} in many familiar non-discrete groups, including the additive group $\Bbb R^d$, the Heisenberg group and the affine
group, all elements except the identity are aperiodic.

In Section 2, we will first demonstrate the result of topologically multiple recurrence for $T_{a,w}$, generated by an aperiodic element $a\in G$, on $L^\Phi(G)$. Following the similar idea and proof, we characterize topologically transitive and mixing weighted translations $T_{a,w}$.
Based on the results in Section 2, a sufficient and necessary condition for $T_{a,w}$ to be chaotic is obtained in Section 3.

\section{recurrence on $L^\Phi(G)$}

First, we have some observations on the norm $N_\Phi$, which will be very useful when pursuing the arguments of the proofs
in the main results. Let $B$ be a Borel set of $G$ with $\lambda(B)>0$, and let $\chi_B$ be the characteristic function
of $B$. Then by a simple computation, we have
$$N_\Phi(\chi_B)=\frac{1}{\Phi^{-1}(\frac{1}{\lambda(B)})}$$
where $\Phi^{-1}(t)$ is the modulus of the preimage of a singleton $t$ under $\Phi$.
Besides, the norm of a function $f\in L^\Phi(G)$ is invariant under the translation by a group element $a\in G$.

\begin{lemma}\label{invariant}
Let $G$ be a locally compact group, and let $a\in G$. Let $\Phi$ be a Young function, and let $f\in L^\Phi(G)$.
Then we have
$$N_\Phi(f)=N_\Phi(f*\delta_a).$$
\end{lemma}
\begin{proof}

By the right invariance of the Haar measure $\lambda$ and the definition of the norm
$$N_\Phi(f)=\inf\left\{k>0:\int_G\Phi\left(\frac{|f|}{k}\right)d\lambda\leq1\right\},$$
we have
\begin{eqnarray*}
N_\Phi(f*\delta_a)&=&\inf\left\{k>0:\int_G\Phi\left(\frac{|f*\delta_a|}{k}\right)d\lambda\leq1\right\}\\
&=&\inf\left\{k>0:\int_G\Phi\left(\frac{1}{k}\left|\int_Gf(xy^{-1})d\delta_a(y)\right|\right)d\lambda(x)\leq1\right\}\\
&=&\inf\left\{k>0:\int_G\Phi\left(\frac{1}{k}|f(xa^{-1})|\right)d\lambda(x)\leq1\right\}\\
&=&\inf\left\{k>0:\int_G\Phi\left(\frac{1}{k}|f(t)|\right)d\lambda(t)\leq1\right\}=N_\Phi(f)
\end{eqnarray*}
where $t=xa^{-1}$ and $d\lambda(t)=d\lambda(xa^{-1})=d\lambda(x)$.
\end{proof}

Now we are ready to prove the result of topologically multiple recurrence on $L^\Phi(G)$.

\begin{theorem}\label{tmr}
Let $G$ be a locally compact group, and let $a\in G$ be an aperiodic element. Let $w$ be a weight on $G$, and let
$\Phi$ be a Young function. Let $T_{a,w}$ be a weighted translation on $L^\Phi(G)$. Then the following conditions are equivalent.
\begin{enumerate}
\item[{\rm(i)}] $T_{a,w}$ is topologically multiply recurrent on $L^\Phi(G)$.
\item[{\rm(ii)}] For each $L\in \Bbb{N}$ and each compact subset $K \subset G$ with $\lambda(K) >0$, there is a
sequence of Borel sets $(E_{k})$ in $K$ such that
$\displaystyle\lambda(K) = \lim_{k \rightarrow \infty}
\lambda(E_{k})$ and both sequences {\rm(}for $1\leq l \leq L${\rm)}
$$\varphi_{ln}:=\prod_{j=1}^{ln}w\ast\delta_{a^{-1}}^{j}\ \ \ \  and \ \ \
\ \ {\widetilde \varphi_{ln}}:=\left(\prod_{j=0}^{ln-1}w\ast\delta_{a}^{j}\right)^{-1}$$
admit respectively subsequences $(\varphi_{ln_k})$ and $({\widetilde \varphi_{ln_k}})$ satisfying
$$\lim_{k\rightarrow \infty}\|\varphi_{ln_{k}}|_{_{E_{k}}}\|_\infty=
\lim_{k\rightarrow \infty}\|{\widetilde \varphi_{ln_{k}}}|_{_{E_{k}}}\|_\infty=0.$$
\end{enumerate}
\end{theorem}
\begin{proof}
(i) $\Rightarrow$ (ii). Let $T_{a,w}$ be topologically multiply recurrent. Let
$K\subset G$ be a compact set with $\lambda (K)
>0$. Let $\varepsilon\in(0,1)$. By aperiodicity of $a$, there is some $M$ such that $K\cap Ka^{\pm n}=\emptyset$
for $n> M$. Let $\chi_{K}\in L^\Phi(G)$ be the characteristic
function of $K$.

Let $U=\{g\in L^\Phi(G):N_\Phi(g-\chi_{K})<\varepsilon^2\}$. Given some $L\in\Bbb N$,
there exists
$m>M$ such that
$$U\cap T_{a,w}^{-m}U\cap T_{a,w}^{-2m}U\cap\cdot\cdot\cdot\cap T_{a,w}^{-Lm}U\neq\emptyset$$
by the assumption of topologically multiple recurrence.
Hence there exists $f\in L^\Phi(G)$ such that
$$N_\Phi(f-\chi_{K}) < \varepsilon^2 \quad \mbox{and} \quad N_\Phi(T_{a,w}^{lm}f-\chi_{K})< \varepsilon^2$$
for $l=1,2,\ldots,L$.
Let $$A=\{x\in K:|f(x)-1|\geq\varepsilon\}.$$
Then
$$|f(x)|>1-\varepsilon \qquad ( x\in K\setminus A)\qquad \mbox{and}\qquad \lambda(A)<\frac{1}{\Phi(\frac{1}{\varepsilon})}$$
by
\begin{eqnarray*}
\varepsilon^2 &>& N_\Phi(f-\chi_K)\\
&\geq& N_\Phi\left(\chi_K(f-1)\right)\\
&\geq& N_\Phi\left(\chi_{A}(f-1)\right)\\
&\geq& N_\Phi(\chi_{A}\varepsilon)\\
&=& \frac{\varepsilon}{\Phi^{-1}(\frac{1}{\lambda(A)})}
\end{eqnarray*}
which implies $\lambda(A)<\frac{1}{\Phi(\frac{1}{\varepsilon})}$.
Let $$B_{l,m}=\{x\in K:|T_{a,w}^{lm}f(x)-1|\geq\varepsilon\}.$$
Then
$$|T_{a,w}^{lm}f(x)|>1-\varepsilon \qquad ( x\in K\setminus B_{l,m})\qquad \mbox{and}\qquad \lambda(B_{l,m})<\frac{1}{\Phi(\frac{1}{\varepsilon})}$$
by the following estimate
\begin{eqnarray*}
\varepsilon^2 &>& N_\Phi(T_{a,w}^{lm}f-\chi_K)\\
&\geq& N_\Phi\left(\chi_K(T_{a,w}^{lm}f-1)\right)\\
&\geq& N_\Phi\left(\chi_{B_{l,m}}(T_{a,w}^{lm}f-1)\right)\\
&\geq& N_\Phi(\chi_{B_{l,m}}\varepsilon)\\
&=& \frac{\varepsilon}{\Phi^{-1}(\frac{1}{\lambda(B_{l,m})})}.
\end{eqnarray*}
Let $$C_{l,m}=\{x\in K\setminus A:\varphi_{lm}(x)\geq \varepsilon\}.$$
Then
$$\varphi_{lm}(x) < \varepsilon\qquad(x\in K\setminus (A\cup C_{lm}))$$
and $\lambda(C_{l,m})<\frac{1}{\Phi(\frac{1-\varepsilon}{\varepsilon})}$.
Indeed, by Lemma \ref{invariant}, the right invariance of the Haar measure $\lambda$ and $K\cap Ka^{\pm m}=\emptyset$, we have
\begin{eqnarray*}
\varepsilon^2 &>& N_\Phi(T_{a,w}^{lm}f-\chi_K)\\
&\geq& N_\Phi\left(\chi_{C_{l,m}a^{lm}}(T_{a,w}^{lm}f)\right)\\
&=& N_\Phi\left(\chi_{C_{l,m}a^{lm}}(\prod_{j=0}^{lm-1}w*\delta_{a}^j)(f*\delta_{a^{lm}})\right)\\
&=& N_\Phi\left(\chi_{C_{l,m}}(\prod_{j=1}^{lm}w*\delta_{a^{-1}}^j)f\right)\\
&=& N_\Phi(\chi_{C_{l,m}}\varphi_{lm}f)\\
&>& \frac{\varepsilon(1-\varepsilon)}{N^{-1}_\Phi(\frac{1}{\lambda(C_{l,m})})}
\end{eqnarray*}
which implies
$\lambda(C_{l,m})<\frac{1}{\Phi(\frac{1-\varepsilon}{\varepsilon})}$.
Let $$D_{l,m}=\{x\in K\setminus B_{l,m}:\widetilde{\varphi}_{lm}(x) \geq\varepsilon\}.$$
Then
$$\widetilde{\varphi}_{lm}(x) < \varepsilon\qquad(x\in K\setminus (B_{l,m}\cup D_{l,m}))$$
and $\lambda(D_{l,m})<\frac{1}{\Phi(\frac{1-\varepsilon}{\varepsilon})}$. Again, we have

\begin{eqnarray*}
\varepsilon^2 &>& N_\Phi(f-\chi_K)\\
&\geq& N_\Phi\left(\chi_{D_{l,m}a^{-lm}}(S_{a,w}^{lm}T_{a,w}^{lm}f)\right)\\
&=& N_\Phi\left(\chi_{D_{l,m}a^{-lm}}(\prod_{j=1}^{lm}w*\delta_{a^{-1}}^j)^{-1}((T_{a,w}^{lm}f)*\delta_{a^{-lm}})\right)\\
&=& N_\Phi\left(\chi_{D_{l,m}}(\prod_{j=0}^{lm-1}w*\delta_{a}^j)^{-1}(T_{a,w}^{lm}f)\right)\\
&=& N_\Phi\left(\chi_{D_{l,m}}\widetilde{\varphi}_{lm}(T_{a,w}^{lm}f)\right)\\
&>& \frac{\varepsilon(1-\varepsilon)}{N^{-1}_\Phi(\frac{1}{\lambda(D_{l,m})})}
\end{eqnarray*}
which implies
$\lambda(D_{l,m})<\frac{1}{\Phi(\frac{1-\varepsilon}{\varepsilon})}.$
Let $$E_{m}=(K\setminus A) \setminus \displaystyle\bigcup_{l=1}^L (B_{l,m}\cup C_{l,m}\cup D_{l,m}).$$
Then we have
$$\lambda(K\setminus E_{m})<\frac{1+L}{\Phi(\frac{1}{\varepsilon})}+\frac{2L}{\Phi(\frac{1-\varepsilon}{\varepsilon})}$$
and
$$\|\varphi_{lm}|_{_{E_{m}}}\|_\infty<\varepsilon, \qquad \|{\widetilde \varphi_{lm}}|_{_{E_{m}}}\|_\infty<\varepsilon,$$
which implies condition (ii) together with the fact $\lim_{t\rightarrow\infty}\Phi(t)=\infty$.

(ii) $\Rightarrow$ (i). We show that $T_{a,w}$ is topologically multiply recurrent.
Let $U$ be a non-empty open subset of $L^\Phi(G)$.
Since the space $C_c(G)$ of continuous functions on $G$ with
compact support is dense in $L^\Phi(G)$, we can pick $f\in C_c(G)$
with $f\in U$. Let $K$ be the compact
support of $f$. Given some $L\in \Bbb{N}$, let $E_k\subset K$ and the sequences $(\varphi_{ln}), ({\widetilde \varphi_{ln}})$
satisfy condition (ii).

By aperiodicity of $a$, there exists $M \in \Bbb{N}$ such that
$K\cap Ka^{\pm n}=\emptyset$ for all $n> M$.
By condition (ii), there exists $M' \in \Bbb{N}$ such that $n_k >M$ and
$\varphi_{ln_{k}}, \widetilde{\varphi}_{ln_{k}} <\frac{\frac{1}{2^k}}{\| f \|_{\infty}}$ for $k> M'$. Hence, using Lemma \ref{invariant}, we obtain
\begin{eqnarray*}
&&N_{\Phi}\left(T_{a,w}^{ln_k}(f\chi_{E_k})\right)\\
&=&N_{\Phi}\left((\prod_{j=0}^{ln_k-1}w*\delta_{a}^j)(f*\delta_{a^{ln_k}})(\chi_{E_k}*\delta_{a^{ln_k}})\right)\\
&=&N_{\Phi}\left((\prod_{j=1}^{ln_k}w*\delta_{a^{-1}}^j)f\chi_{E_k}\right)\\
&=&N_{\Phi}\left(\varphi_{ln_k}f\chi_{E_k}\right)\\
&<&\frac{\|f\|_{\infty}\frac{\frac{1}{2^k}}{\|f\|_{\infty}}}{\Phi^{-1}(\frac{1}{\lambda(E_k)})}\rightarrow 0
\end{eqnarray*}
as $k\rightarrow \infty$ for $1\leq l\leq L$. Similarly, by the sequence $(\widetilde \varphi_{ln_k})$,
\begin{eqnarray*}
&&\lim_{k\rightarrow \infty}N_{\Phi}\left(S_{a,w}^{ln_k}(f\chi_{E_k})\right)\\
&=&\lim_{k\rightarrow \infty}N_{\Phi}\left((\prod_{j=1}^{ln_k}w*\delta_{a^{-1}}^j)^{-1}(f*\delta_{a^{-ln_k}})(\chi_{E_k}*\delta_{a^{-ln_k}})\right)\\
&=&\lim_{k\rightarrow \infty}N_{\Phi}\left((\prod_{j=0}^{ln_k-1}w*\delta_{a}^j)^{-1}f\chi_{E_k}\right)\\
&=&\lim_{k\rightarrow \infty}N_{\Phi}\left(\widetilde{\varphi}_{ln_k}f\chi_{E_k}\right)=0
\end{eqnarray*}
for $l=1,2,...,L$.

Now we are ready to achieve our goal. For each $k\in \Bbb{N}$, we let
$$v_k=f\chi_{E_k}+S_{a,w}^{n_k}(f\chi_{E_k})+S_{a,w}^{2n_k}(f\chi_{E_k})+\cdot\cdot\cdot+S_{a,w}^{Ln_k}(f\chi_{E_k}).$$
Then
$$N_{\Phi}(v_k-f)\leq N_{\Phi}(f\chi_{K\setminus E_k})+\sum_{l=1}^{L}N_{\Phi}\left(S_{a,w}^{ln_k}(f\chi_{E_k})\right)$$
and
\begin{eqnarray*}
N_{\Phi}(T_{a,w}^{ln_k}v_k-f)&\leq& N_{\Phi}\left(T_{a,w}^{ln_k}(f\chi_{E_k})\right)+N_{\Phi}\left(T_{a,w}^{(l-1)n_k}(f\chi_{E_k})\right)
+\cdot\cdot\cdot+N_{\Phi}\left(T_{a,w}^{n_k}(f\chi_{E_k})\right)\\
&+& N_{\Phi}(f\chi_{K\setminus E_k})+N_{\Phi}\left(S_{a,w}^{n_k}(f\chi_{E_k})\right)
+\cdot\cdot\cdot+N_{\Phi}\left(S_{a,w}^{(L-l)n_k}(f\chi_{E_k})\right),
\end{eqnarray*}
which implies $\displaystyle\lim_{k\rightarrow\infty}N_{\Phi}(v_k-f)=\displaystyle\lim_{k\rightarrow\infty}N_{\Phi}(T_{a,w}^{ln_k}v_k-f)=0$. Hence
$$U\cap T_{a,w}^{-n_k}U\cap T_{a,w}^{-2n_k}U\cap\cdot\cdot\cdot\cap T_{a,w}^{-Ln_k}U\neq\emptyset.$$
\end{proof}

As in \cite[Example 2.2, 2.4, 2.5]{chen16}, it is not difficult to find the weight satisfying the above weight condition
on various locally compact groups. For completeness, we include one example for Heisenberg groups only.

\begin{example}\label{example}
Let $$G=\Bbb{H}:=\left\{\left(\begin{matrix}
1 & x & z\\
0 & 1 & y\\
0 & 0 & 1
\end{matrix}\right):x,y,z\in \Bbb{R}\right\}$$
be the Heisenberg group which is neither abelian nor compact. For
convenience, an element in $G$ is written as $(x,y,z)$. Let $(x,y,z),(x',y',z')\in \Bbb{H}$.
Then the multiplication is given by $$(x,y,z)\cdot(x',y',z')=(x+x',y+y',z+z'+xy')$$ and
$$(x,y,z)^{-1}=(-x,-y,xy-z).$$

Let $a=(3,0,2)$ and $w$ be a weight on $\Bbb{H}$.
Then $a^{-1}=(-3,0,-2)$ and the weighted translation $T_{(3,0,2),w}$
on $L^\Phi(\Bbb{H})$ is given by
$$T_{(3,0,2),w}f(x,y,z)=w(x,y,z)f(x-3,y,z-2)\qquad(f\in L^p(\Bbb{H})).$$
By Theorem \ref{tmr}, the operator $T_{(3,0,2),w}$
is topologically multiply recurrent if given $\varepsilon>0$, some $L\in\Bbb{N}$ and a compact subset $K$ of $\Bbb{H}$,
there exists a positive integer $n$ such that for $1\leq l \leq L$ and $x\in K$, we have
$$\varphi_{ln}(x,y,z)=\prod_{s=1}^{ln}w\ast\delta_{_{(3,0,2)^{-1}}}^{s}(x,y,z)=\prod_{s=1}^{ln}w(x+3s,y,z+2s)<\varepsilon$$
and
$${\widetilde\varphi_{ln}}^{^{-1}}(x,y,z)=\prod_{s=0}^{ln-1}w\ast\delta_{_{(3,0,2)}}^{s}(x)=\prod_{s=0}^{ln-1}w(x-3s,y,z-2s)>\frac{1}{\varepsilon}.$$
One can obtain the required weight condition by defining $w:\Bbb{H}\rightarrow (0,\infty)$ as follows:
$$w(x,y,z)=\left\{
\begin{array}{ll}
\frac{1}{2}& \mbox{ if } z\geq 1\\\\
\frac{1}{2^z}& \mbox{ if } -1<z< 1\\\\
2 & \mbox{ if } z\leq-1.
\end{array}
\right.$$
\end{example}

If $L=1$ in the Theorem \ref{tmr}, then the characterization for recurrence follows immediately.

\begin{corollary}\label{recurrence}
Let $G$ be a locally compact group, and let $a\in G$ be an aperiodic element. Let $w$ be a weight on $G$, and let
$\Phi$ be a Young function. Let $T_{a,w}$ be a weighted translation on $L^\Phi(G)$. Then the following conditions are equivalent.
\begin{enumerate}
\item[{\rm(i)}] $T_{a,w}$ is recurrent on $L^\Phi(G)$.
\item[{\rm(ii)}] For each compact subset $K \subset G$ with $\lambda(K) >0$, there is a
sequence of Borel sets $(E_{k})$ in $K$ such that
$\displaystyle\lambda(K) = \lim_{k \rightarrow \infty}
\lambda(E_{k})$ and both sequences
$$\varphi_{n}:=\prod_{j=1}^{n}w\ast\delta_{a^{-1}}^{j}\ \ \ \  and \ \ \
\ \ {\widetilde \varphi_{n}}:=\left(\prod_{j=0}^{n-1}w\ast\delta_{a}^{j}\right)^{-1}$$
admit respectively subsequences $(\varphi_{n_k})$ and $({\widetilde \varphi_{n_k}})$ satisfying
$$\lim_{k\rightarrow \infty}\|\varphi_{n_{k}}|_{_{E_{k}}}\|_\infty=
\lim_{k\rightarrow \infty}\|{\widetilde \varphi_{n_{k}}}|_{_{E_{k}}}\|_\infty=0.$$
\end{enumerate}
\end{corollary}

Following the similar argument as in the proof of Theorem \ref{tmr}, one can characterize topological transitivity (hypercyclicity)
and mixing for weighted translations $T_{a,w}$ on $L^\Phi(G)$,
which recovers the result of \cite[Theorem 2.3, Corollary 2.5]{aa17} where a different approach was applied.

\begin{corollary}\label{transitivity}
Let $G$ be a locally compact group, and let $a\in G$ be an aperiodic element. Let $w$ be a weight on $G$, and let
$\Phi$ be a Young function. Let $T_{a,w}$ be a weighted translation on $L^\Phi(G)$. Then the following conditions are equivalent.
\begin{enumerate}
\item[{\rm(i)}] $T_{a,w}$ is topologically transitive on $L^\Phi(G)$.
\item[{\rm(ii)}] For each compact subset $K \subset G$ with $\lambda(K) >0$, there is a
sequence of Borel sets $(E_{k})$ in $K$ such that
$\displaystyle\lambda(K) = \lim_{k \rightarrow \infty}
\lambda(E_{k})$ and both sequences
$$\varphi_{n}:=\prod_{j=1}^{n}w\ast\delta_{a^{-1}}^{j}\ \ \ \  and \ \ \
\ \ {\widetilde \varphi_{n}}:=\left(\prod_{j=0}^{n-1}w\ast\delta_{a}^{j}\right)^{-1}$$
admit respectively subsequences $(\varphi_{n_k})$ and $({\widetilde \varphi_{n_k}})$ satisfying
$$\lim_{k\rightarrow \infty}\|\varphi_{n_{k}}|_{_{E_{k}}}\|_\infty=
\lim_{k\rightarrow \infty}\|{\widetilde \varphi_{n_{k}}}|_{_{E_{k}}}\|_\infty=0.$$
\end{enumerate}
\end{corollary}
\begin{proof}
(i) $\Rightarrow$ (ii). Let $T_{a,w}$ be topologically transitive.
By the assumption of transitivity, there exist $f\in L^\Phi(G)$ and some $m\in\Bbb{N}$ such that
$$N_\Phi(f-\chi_{K}) < \varepsilon^2 \quad \mbox{and} \quad N_\Phi(T_{a,w}^{m}f-\chi_{K})< \varepsilon^2.$$
Following the similar arguments as in the proof of Theorem \ref{tmr}, one can obtain
$$\|\varphi_{m}|_{_{E_{m}}}\|_\infty<\varepsilon, \qquad \|{\widetilde \varphi_{m}}|_{_{E_{m}}}\|_\infty<\varepsilon,$$
which implies condition (ii).

(ii) $\Rightarrow$ (i).
Let $U,V$ be non-empty open subsets of $L^\Phi(G)$.
Then we can pick $f,g\in C_c(G)$ with $f\in U$ and $g\in V$. Let $K$ be the compact
support of $f$ and $g$, and let
$$v_k=f\chi_{E_k}+S_{a,w}^{n_k}(g\chi_{E_k}).$$
Then
$$N_{\Phi}(v_k-f)\leq N_{\Phi}(f\chi_{K\setminus E_k})+S_{a,w}^{n_k}(g\chi_{E_k})\rightarrow0$$
and
$$N_{\Phi}(T_{a,w}^{n_k}v_k-g) \leq N_{\Phi}(T_{a,w}^{n_k}(f\chi_{E_k}))+N_{\Phi}(g\chi_{K\setminus E_k})\rightarrow0,$$
as $k\rightarrow\infty$. Hence
$$T_{a,w}^{n_k}U\cap V\neq\emptyset.$$
\end{proof}

\begin{remark}
By Corollary \ref{recurrence} and Corollary \ref{transitivity}, one can deduce that a weighted translation operator $T_{a,w}$ on $L^\Phi(G)$ is topologically transitive (hypercyclic) if, and only if, it is recurrent.
\end{remark}

\begin{corollary}\label{mixing}
Let $G$ be a locally compact group, and let $a\in G$ be an aperiodic element. Let $w$ be a weight on $G$, and let
$\Phi$ be a Young function. Let $T_{a,w}$ be a weighted translation on $L^\Phi(G)$. Then the following conditions are equivalent.
\begin{enumerate}
\item[{\rm(i)}] $T_{a,w}$ is topologically mixing on $L^\Phi(G)$.
\item[{\rm(ii)}] For each compact subset $K \subset G$ with $\lambda(K) >0$, there is a
sequence of Borel sets $(E_n)$ in $K$ such that
$\displaystyle\lambda(K) = \lim_{n \rightarrow \infty}
\lambda(E_n)$ and both sequences
$$\varphi_{n}:=\prod_{j=1}^{n}w\ast\delta_{a^{-1}}^{j}\ \ \ \  and \ \ \
\ \ {\widetilde \varphi_{n}}:=\left(\prod_{j=0}^{n-1}w\ast\delta_{a}^{j}\right)^{-1}$$
satisfy
$$\lim_{n\rightarrow \infty}\|\varphi_{n}|_{_{E_{n}}}\|_\infty=\lim_{n\rightarrow \infty}\|{\widetilde \varphi_{n}}|_{_{E_{n}}}\|_\infty=0.$$
\end{enumerate}
\end{corollary}
\begin{proof}
The proof is similar to that of Corollary \ref{transitivity} by using the full sequence $(n)$ instead of subsequence $(n_k)$.
\end{proof}

\section{chaotic condition}

Applying Corollary \ref{transitivity}, in this section, we will give a sufficient and necessary condition for $T_{a,w}$
to be chaotic on the Orlicz space $L^\Phi(G)$. This result and Corollary \ref{mixing} entail one further consequence, that is, chaos and mixing imply topological multiple recurrence on $T_{a,w}$ in our case.

\begin{theorem}\label{chaos}
Let $G$ be a locally compact group, and let $a\in G$ be an aperiodic element. Let $w$ be a weight on $G$, and let
$\Phi$ be a Young function. Let $T_{a,w}$ be a weighted translation on $L^\Phi(G)$, and let $\cal P(T_{a,w})$ be the set of periodic elements of $T_{a,w}$. Then the following conditions are equivalent.
\begin{enumerate}
\item[{\rm(i)}] $T_{a,w}$ is chaotic on $L^\Phi(G)$.
\item[{\rm(ii)}] $\cal P(T_{a,w})$ is dense in $L^\Phi(G)$.
\item[{\rm(iii)}]For each compact subset $K \subseteq G$ with $\lambda(K) >0$, there is a
sequence of Borel sets $(E_{k})$ in $K$ such that
$\displaystyle\lambda(K) = \lim_{k \rightarrow \infty}\lambda(E_k)$, and both sequences
$$\varphi_{n}:=\prod_{j=1}^{n}w\ast\delta_{a^{-1}}^{j}\ \ \ \  and \ \ \ \ \
{\widetilde \varphi_{n}}:=\left(\prod_{j=0}^{n-1}w\ast\delta_{a}^{j}\right)^{-1}$$
admit respectively subsequences $(\varphi_{n_k})$ and $(\widetilde{\varphi}_{n_k})$ satisfying
$$\lim_{k\rightarrow\infty}\left\|\sum_{l=1}^{\infty}\varphi_{ln_k}+\sum_{l=1}^{\infty}\widetilde{\varphi}_{ln_k}\Big|_{E_k}\right\|_{\infty}=0.$$
\end{enumerate}
\end{theorem}

\begin{proof}
We will show (ii) $\Rightarrow$ (iii), and (iii) $\Rightarrow$ (i).

(ii) $\Rightarrow$ (iii). Let $K \subseteq G$ be a compact set with $\lambda(K) >0$.
Since $a$ is aperiodic, there exists $M\in\Bbb{N}$ such that $K\cap Ka^{\pm m}=\emptyset$ for all $m> M$. Let $\chi_K\in L^\Phi(G)$
be the characteristic function of $K$. By the density of $\cal P(T_{a,w})$, we can find a sequence $(f_k)$ of periodic points of $T_{a,w}$ satisfying
$N_{\Phi}(f_k-\chi_{K})<\frac{1}{4^k}$, and a sequence $(n_k)\subset \Bbb{N}$ such that $T_{a,w}^{n_k}f_k=f_k=S_{a,w}^{n_k}f_k$, where we may assume $n_{k+1}>n_k> M$.
Therefore, $Ka^{rn_k}\cap Ka^{sn_k}=\emptyset$ for all $r,s\in \Bbb{Z}$ with $r\neq s$.

Let $A_k=\{x\in K:|f_k(x)-1|\geq\frac{1}{2^k}\}$.
Then
$$|f_k(x)|>1-\frac{1}{2^k} \qquad ( x\in K\setminus A_k)\qquad \mbox{and}\qquad \lambda(A_k)<\frac{1}{\Phi(2^k)}.$$
Indeed,
\begin{eqnarray*}
\frac{1}{4^k} &>& N_\Phi(f_k-\chi_K)\\
&\geq& N_\Phi\left(\chi_K(f_k-1)\right)\\
&\geq& N_\Phi\left(\chi_{A_k}(f_k-1)\right)\\
&\geq& N_\Phi(\chi_{A_k}\frac{1}{2^k})\\
&=& \frac{\frac{1}{2^k}}{\Phi^{-1}(\frac{1}{\lambda(A_k)})}
\end{eqnarray*}
which implies $\lambda(A_k)<\frac{1}{\Phi(2^k)}$.
Let
$$B_k=\left\{x\in K\setminus A_k:\sum_{l=1}^{\infty}\varphi_{ln_k}(x)+\sum_{l=1}^{\infty}\widetilde{\varphi}_{ln_k}(x)\geq\frac{1}{2^k}\right\},$$
and let $E_k=K\setminus (A_k\cup B_k)$.
Then
$$\sum_{l=1}^{\infty}\varphi_{ln_k}(x)+\sum_{l=1}^{\infty}\widetilde{\varphi}_{ln_k}(x)<\frac{1}{2^k} \qquad ( x\in E_k).$$
To complete the proof, we will show
$$\lambda(K\setminus E_k)<\frac{1}{\Phi(2^k)}+\frac{1}{\Phi(2^k-1)}.$$
Again, by Lemma \ref{invariant}, the right invariance of the Haar measure $\lambda$, and $Ka^{rn_k}\cap Ka^{sn_k}=\emptyset$ for $r\neq s$, we arrive at
\begin{eqnarray*}
\frac{1}{4^k} &>& N_\Phi(f_k-\chi_K)\\
&\geq& N_\Phi\left(\chi_{G\setminus K}(f_k-0)\right)\\
&\geq& N_\Phi\left(\sum_{l=1}^{\infty}\chi_{Ka^{ln_k}}f_k+\sum_{l=1}^{\infty}\chi_{Ka^{-ln_k}}f_k\right)\\
&=& N_\Phi\left(\sum_{l=1}^{\infty}\chi_{K}(f_k*\delta_{a^{-ln_k}})+\sum_{l=1}^{\infty}\chi_{K}(f_k*\delta_{a^{ln_k}})\right)\\
&=& N_\Phi\left(\sum_{l=1}^{\infty}\chi_{K}((T_{a,w}^{ln_k}f_k)*\delta_{a^{-ln_k}})
+\sum_{l=1}^{\infty}\chi_{K}((S_{a,w}^{ln_k}f_k)*\delta_{a^{ln_k}})\right)\\
&\geq& N_\Phi\left(\sum_{l=1}^{\infty}\chi_{B_k}\varphi_{ln_k}f+\sum_{l=1}^{\infty}\chi_{B_k}\widetilde{\varphi}_{ln_k}f\right)\\
&>& N_\Phi\left((1-\frac{1}{2^k}){\frac{1}{2^k}}\chi_{B_k}\right)\\
&=& \frac{(1-\frac{1}{2^k}){\frac{1}{2^k}}}{\Phi^{-1}(\frac{1}{\lambda(B_k)})}.
\end{eqnarray*}
Hence $\Phi^{-1}(\frac{1}{\lambda(B_k)})>2^k-1$ which yields that $\lambda(B_k)<\frac{1}{\Phi(2^k-1)}$.
So, together with $\lambda(A_k)<\frac{1}{\Phi(2^k)}$, the estimate $\lambda(K\setminus E_k)<\frac{1}{\Phi(2^k)}+\frac{1}{\Phi(2^k-1)}$
follows.

(iii) $\Rightarrow$ (i). By Corollary \ref{transitivity}, condition (iii) implies $T_{a,w}$ is topologically transitive.
Here we will show ${\cal P}(T_{a,w})$ is dense in $L^\Phi(G)$.
Let $f\in C_c(G)$ with compact support $K\subseteq G$. Then there exist a sequence of Borel sets $(E_k)$ in $K$, and a sequence $(n_k)$
such that $\displaystyle\lambda(K) = \lim_{k \rightarrow \infty}\lambda(E_k)$ and
$$\sum_{l=1}^{\infty}\varphi_{ln_k}(x)+\sum_{l=1}^{\infty}\widetilde{\varphi}_{ln_k}(x)<\frac{1}{2^k}\qquad (x\in E_k).$$
Let
$$v_k:=f\chi_{E_k}+\sum_{l=1}^{\infty}T^{ln_k}_{a,w}(f\chi_{E_k})+\sum_{l=1}^{\infty}S^{ln_k}_{a,w}(f\chi_{E_k}).$$
Then
\begin{eqnarray*}
T_{a,w}^{n_k}v_k&=&T_{a,w}^{n_k}(f\chi_{E_k})+\sum_{l=1}^{\infty}T_{a,w}^{n_k}T^{ln_k}_{a,w}(f\chi_{E_k})
+\sum_{l=1}^{\infty}T_{a,w}^{n_k}S^{ln_k}_{a,w}(f\chi_{E_k})\\
&=&\sum_{l=1}^{\infty}T^{ln_k}_{a,w}(f\chi_{E_k})+f\chi_{E_k}+\sum_{l=1}^{\infty}S^{ln_k}_{a,w}(f\chi_{E_k})=v_k
\end{eqnarray*}
which implies that $v_k\in {\cal P}(T_{a,w})$ for each $k$. Moreover, we note that
$$N_{\Phi}(f\chi_{K\setminus E_k})\leq \frac{\|f\|_{\infty}}{\Phi^{-1}(\frac{1}{\lambda(K\setminus E_k)})}\rightarrow 0$$
as $k\rightarrow\infty$ by the fact $\displaystyle\lambda(K) = \lim_{k \rightarrow \infty}\lambda(E_k)$.
Also, by Lemma \ref{invariant} and the invariance of the Haar measure $\lambda$,
\begin{eqnarray*}
&&N_{\Phi}\left(\sum_{l=1}^{\infty}T_{a,w}^{ln_k}(f\chi_{E_k})+\sum_{l=1}^{\infty}S_{a,w}^{ln_k}(f\chi_{E_k})\right)\\
&=&N_{\Phi}\left(\sum_{l=1}^{\infty}(\prod_{j=0}^{ln_k-1}w*\delta_{a}^j)(f*\delta_{a^{ln_k}})(\chi_{E_k}*\delta_{a^{ln_k}})
+\sum_{l=1}^{\infty}(\prod_{j=1}^{ln_k}w*\delta_{a^{-1}}^j)^{-1}(f*\delta_{a^{-ln_k}})(\chi_{E_k}*\delta_{a^{-ln_k}})\right)\\
&=&N_{\Phi}\left(\sum_{l=1}^{\infty}(\prod_{j=1}^{ln_k}w*\delta_{a^{-1}}^j)f\chi_{E_k}
+\sum_{l=1}^{\infty}(\prod_{j=0}^{ln_k-1}w*\delta_{a}^j)^{-1}f\chi_{E_k}\right)\\
&=&N_{\Phi}\left(\sum_{l=1}^{\infty}\varphi_{ln_k}f\chi_{E_k}+\sum_{l=1}^{\infty}\widetilde{\varphi}_{ln_k}f\chi_{E_k}\right)\\
&<&\frac{\|f\|_{\infty}\frac{1}{2^k}}{\Phi^{-1}(\frac{1}{\lambda(E_k)})}\rightarrow 0
\end{eqnarray*}
as $k\rightarrow\infty$.
Hence we arrive at
$$N_{\Phi}(v_k-f)\leq N_{\Phi}(f\chi_{K\setminus E_k})+N_{\Phi}\left(\sum_{l=1}^{\infty}T_{a,w}^{ln_k}(f\chi_{E_k})+\sum_{l=1}^{\infty}S_{a,w}^{ln_k}(f\chi_{E_k})\right)\rightarrow 0$$
as $k\rightarrow \infty$, which says $v_k\rightarrow f$ as $k\rightarrow \infty$. Combing all these, $\cal P(T_{a,w})$ is dense in $L^\Phi(G)$.
\end{proof}

\begin{example}
Let $G$ be the Heisenberg group $\Bbb{H}$, and let $T_{(3,0,2),w}$ be the weighted translation
on $L^\Phi(\Bbb{H})$ as defined in Example \ref{example}. Let
$$w(x,y,z)=\left\{
\begin{array}{ll}
\frac{1}{2}& \mbox{ if } z\geq 1\\\\
\frac{1}{2^z}& \mbox{ if } -1<z< 1\\\\
2 & \mbox{ if } z\leq-1.
\end{array}
\right.$$
Then the operator $T_{(3,0,2),w}$ is chaotic. Indeed,
given a compact subset $K$ of $\Bbb{H}$,
without loss of generality, we may assume $(x,y,0)\in K$.
Then
\begin{eqnarray*}
&&\sum_{l=1}^{\infty}\varphi_{ln}(x,y,0)+\sum_{l=1}^{\infty}\widetilde{\varphi}_{ln}(x,y,0)\\
&=&\sum_{l=1}^{\infty}\prod_{s=1}^{ln}w(x+3s,y,0+2s)+\sum_{l=1}^{\infty}\frac{1}{\prod_{s=0}^{ln-1}w(x-3s,y,0-2s)}\\
&=&\sum_{l=1}^{\infty}\frac{1}{2^{ln}}+\sum_{l=1}^{\infty}\frac{1}{2^{ln-1}}=\frac{3}{2^n-1}\rightarrow0
\end{eqnarray*}
as $n\rightarrow\infty$.
\end{example}

\begin{corollary}
Let $G$ be a locally compact group, and let $a\in G$ be an aperiodic element. Let $w$ be a weight on $G$, and let
$\Phi$ be a Young function. Let $T_{a,w}$ be a weighted translation on $L^\Phi(G)$.
If $T_{a,w}$ is chaotic or mixing, then it is topologically multiply recurrent.
\end{corollary}
\begin{proof}
The result follows by Corollary \ref{mixing} and Theorem \ref{chaos}.
\end{proof}

\begin{remark}
We note that there exists a weighted shift which is topologically multiply recurrent but is neither chaotic nor topologically mixing in \cite{bg07}.
\end{remark}

\vspace{.1in}
\end{document}